\documentclass{amsart}
\usepackage{amsthm,amsmath,amssymb,amsfonts}

\newtheorem{theorem}{Theorem}[section]
\newtheorem{lemma}[theorem]{Lemma}
\newtheorem{proposition}[theorem]{Proposition}
\newtheorem{corollary}[theorem]{Corollary}

\theoremstyle{definition}
\newtheorem{definition}[theorem]{Definition}
\newtheorem{example}[theorem]{Example}

\theoremstyle{remark}
\newtheorem{remark}[theorem]{Remark}

\numberwithin{equation}{section}

\begin{document}
\title[Reduced Archimedean skew polynomial rings and power series rings]{Reduced Archimedean skew polynomial rings and skew power series rings}
\author{Ryszard Mazurek}
\address{Faculty of Computer
Science, Bialystok University of Technology, Wiejska 45A, 15--351
Bia{\l}ystok, Poland} \email{r.mazurek@pb.edu.pl}

%\date{19.09.2020}

\begin{abstract}
We characterize skew polynomial rings and skew power series rings that are reduced and right or left Archimedean.
\end{abstract}

\maketitle

\section{Introduction}
We begin by recalling what Archimedean rings are, in which context they appeared, and presenting shortly the contents of this paper. In this paper rings are associative with unity, but not necessarily commutative. For a ring $R$, $R[x]$ denotes the polynomial ring, $R[[x]]$ the power series ring, $U(R)$ the group of units of $R$, $J(R)$ the Jacobson radical of $R$, and if $R$ is a commutative domain, then $Q(R)$ denotes the quotient field of $R$.

Let $D$ be a commutative domain and let $K$ be the quotient field of $D$. It is easy to see that the polynomial rings $D[x]$ and $K[x]$ have the same quotient field, i.e., $Q(D[x]) = Q(K[x])$. However, as observed by Gilmer in \cite{Gilmer}, the quotient fields of the power series rings $D[[x]]$ and $K[[x]]$ need not coincide. For instance, it follows from Gilmer's theorem \cite[Theorem 1]{Gilmer} that if
$$\text{$\bigcap_{n \in \mathbb{N}} a^nD = 0$ for some $a \in D \setminus \{0\}$},$$
then $Q(D[[x]]) \neq Q(K[[x]])$. In \cite{Sheldon} Sheldon generalized Gilmer's result by repla\-cing~$K$ with the fraction ring $D_S$ of $D$ with respect to a multiplicative subset $S$ of~$D$. In \cite[Corollary 3.9]{Sheldon} he proved that $Q(D[[x]]) \neq Q(D_S[[x]])$ holds for every $S \not\subseteq U(D)$ if and only~if
\begin{equation}\label{3001}
\bigcap_{n \in \mathbb{N}} a^nD = 0 \text{ for every } a \in D \setminus U(D).
\end{equation}
In \cite[Definitions 3.6]{Sheldon}, Sheldon coined the name an {\it Archimedean domain} for a commutative domain that satisfies condition (\ref{3001}), motivating the use of the word ``Archi\-me\-dean" by the strong parallel from the theory of Archimedean partially ordered groups. The class of Archimedean domains includes, e.g., ACCP domains, Mori domains, one-dimensional domains and completely integrally closed domains (see \cite{GR}). Archimedean domains were investigated in many papers, e.g. in \cite{AAZ, BD, Coykendall, Dobbs, Dumitrescu, GR1990, GRII, Heinzer}. 

In \cite{Nasr} the concept of a commutative Archimedean domain was extended to noncommutative domains and in \cite{MPP} a more general notion of a right (left) Archi\-me\-de\-an ring was introduced as follows.

\begin{definition} A ring $R$ is said to be {\it right} (resp. {\it left}) {\it Archimedean} if
$$\text{$\bigcap_{n \in \mathbb{N}} Ra^n = 0$ (resp. $\bigcap_{n \in \mathbb{N}} a^nR = 0$) for any $a \in R \setminus U(R)$.}$$
The notions of a right Archimedean ring and a left Archimedean ring are independent, even in the class of domains, as \cite[Example 5.8]{RMArchim} shows.
\end{definition}

Let $R$ be a ring and $\alpha$ an endomorphism of $R$. The \emph{skew polynomial ring} $R[x; \alpha]$ (resp. the \emph{skew power series ring} $R[[x; \alpha]]$) consists of polynomials (resp. power series) over $R$ in the variable $x$ with coefficients written on the left of powers of $x$, with term-wise addition and with multiplication subject to the relation $xa = \alpha (a)x$ for $a \in R$. Elements of the set $R \setminus U(R)$ are called \emph{nonunits} of $R$, and we say that $\alpha$ \emph{preserves nonunits} of $R$ if $\alpha(R \setminus U(R)) \subseteq R \setminus U(R)$.

In \cite{Nasr}, Nasr-Isfahani obtained the following characterizations of skew polynomial rings and skew power series rings that are right or left Archimedean domains. 

\begin{theorem} \label{2721} {\rm (Nasr-Isfahani \cite[Theorems 2.9, 2.11]{Nasr})} 
Let $R$ be a ring and $\alpha$ an endomorphism of $R$. 
\begin{enumerate}
\item[(a)] The following conditions are equivalent:
\begin{enumerate}
\item[(i)] $R[x; \alpha]$ is a right Archimedean domain.
\item[(ii)] $R[[x; \alpha]]$ is a right Archimedean domain.
\item[(iii)] $R$ is a right Archimedean domain, $\alpha$ is injective and $\alpha$ preserves non\-units of $R$.
\end{enumerate}
\item[(b)] The following conditions are equivalent:
\begin{enumerate}
\item[(i)] $R[x; \alpha]$ is a left Archimedean domain.
\item[(ii)] $R[[x; \alpha]]$ is a left Archimedean domain.
\item[(iii)] $R$ is a left Archimedean domain and $\alpha$ is injective.
\end{enumerate}
\end{enumerate}
\end{theorem}

Recall that a ring $R$ is said to be \emph{reduced} if $R$ has no nonzero nilpotent elements; these rings are a natural generalization of domains. Recently in \cite[Theorem 2.2]{MPQ}, reduced skew power series rings $R[[x; \alpha]]$ which are right (or left) Archimedean were characterized under the additional assumption that the ring $R$ satisfies ACC on annihilators and $\alpha$ is a surjective and rigid endomorphism of $R$.

In this paper we characterize, in full generality (i.e., with no initial conditions on $R$ or $\alpha$), the skew polynomial rings $R[x; \alpha]$ and skew power series rings $R[[x; \alpha]]$ that are reduced and right or left Archimedean. To get the results we use methods that are different from those in \cite{MPQ}.

The paper is organized as follows. In Section \ref{2} we prove some general results on one-sided Archimedean rings. In Section \ref{3} we characterize skew polynomial rings which are reduced and right or left Archimedean. Surprisingly, these rings are necessarily domains, as we will see in Theorems \ref{1130} and \ref{1144}. The main results of Section \ref{4} are Theorems \ref{1133} and \ref{1131} in which we characterize skew power series rings that are reduced and right (resp. left) Archimedean. In the same section we also provide an example of a right and left Archimedean skew power series ring which is reduced but not a domain (see Example \ref{7700}).

\section{Preliminaries} \label{2}

In this short section we collect some properties of Archimedean rings, which will be used in later parts of the paper. We start with an easy observation on passing the right (resp. left) Archimedean condition from a ring to its subrings (see \cite[Proposition 5.1]{RMArchim}; cf. \cite[Proposition 3.5]{MPP}).

\begin{proposition} \label{5500} Let $A$ be a subring of a ring $B$ such that $A \cap U(B) \subseteq U(A)$. If $B$ is right (resp. left) Archimedean, then so is $A$.
\end{proposition}

An element $a$ of a ring $R$ is called a \emph{right} (resp. \emph{left}) \emph{zero-divisor} of $R$ if there exists $b \in R \setminus \{ 0 \}$ such that $ba = 0$ (resp. $ab = 0$). Recall that a ring $R$ is said to be \emph{Dedekind-finite} if for any $a, b \in R$, $ab = 1$ implies $ba = 1$.

\begin{lemma} \label{2211}
If a ring $R$ is right (resp. left) Archimedean, then
\begin{enumerate}
\item[(a)] For any $a, b, c\in R$, if $a = bac$ and $a \neq 0$ then $c \in U(R)$ (resp. $b \in U(R)$).
\item[(b)] All right (resp. left) zero-divisors of $R$ belong to $J(R)$.
\item[(c)] $0$ and $1$ are the only idempotents of $R$.
\item[(d)] $R$ is Dedekind-finite.
\end{enumerate}
\end{lemma}

\begin{proof} We consider only the case where $R$ is right Archimedean. The left Archi\-me\-de\-an case follows by analogous arguments.

(a): If $a = bac$, then $a = b(bac)c = b^2ac^2 = b^2(bac)c^2 = b^3ac^3$, and continuing in this way we obtain $a = b^nac^n$ for any $n \in \mathbb{N}$. Hence $a \in \bigcap_{n \in \mathbb{N}} Rc^n$, and since $R$ is right Archimedean and $a \neq 0$, $c \in U(R)$ follows.

(b): Suppose, for contradiction, that there exists a right zero-divisor $c \in R$ such that $c \not\in J(R)$. Then $ac = 0$ for some $a \in R \setminus \{ 0 \}$ and there exists a maximal right ideal $M$ of $R$ such that $c \not\in M$. The maximality of $M$ implies $cR + M = R$ and thus $cr + m = 1$ for some $r \in R$ and $m \in M$. Now, $a = a(cr + m) = (ac)r + am = am$, and since $a \neq 0$, (a) implies $m \in U(R)$, a contradiction.

(c) follows immediately from (a).

(d) is a direct consequence of (c).
\end{proof}

\begin{remark} 
It is an immediate consequence of Lemma \ref{2211}(b) that each semi\-pri\-mi\-tive ring (i.e., a ring $R$ with $J(R) = 0$) which is right or left Archimedean is automatically a domain. From this observation it follows that a von Neumann regular ring $R$ is right or left Archimedean if and only if  $R$ is a division ring.
\end{remark}

\section{Reduced Archimedean skew polynomial rings} \label{3}

In this section we will characterize reduced skew polynomial rings that are right or left Archimedean. We start with the following observation.

\begin{proposition} \label{1987}
Let $R$ be a ring and $\alpha$ an endomorphism of $R$ such that the skew polynomial ring $R[x; \alpha]$ is right (resp. left) Archimedean. If $f$ is a right (resp. left) zero-divisor of $R[x; \alpha]$, then $fx$ is nilpotent.
\end{proposition}

\begin{proof} Assume $R[x; \alpha]$ is right (resp. left) Archimedean and let $f$ be a right (resp. left) zero-divisor of $R[x; \alpha]$. Then $f \in J(R[x; \alpha])$ by Lemma \ref{2211}(b) and thus $1 + fx \in U(R[x; \alpha])$. Since $R[x; \alpha]$ is a subring of the skew power series ring $R[[x; \alpha]]$ and $1 + fx$ is invertible in $R[[x; \alpha]]$ with\begin{equation}\label{9328}
(1 + fx)^{-1} = 1 - fx + (fx)^2 - (fx)^3 + \cdots,
\end{equation}
the power series (\ref{9328}) has to be also the inverse of $1 + fx$ in $R[x; \alpha]$ and thus the power series (\ref{9328}) is, in fact, a polynomial. Hence $fx$ is nilpotent.
\end{proof}

For a ring $R$, let $Z_r(R)$ (resp. $Z_l(R)$) denote the set of right (resp. left) zero-divisors of $R$, and let $N(R)$ denote the set of nilpotent elements of $R$.

\begin{corollary} \label{8104}
Let $R$ be a ring such that the polynomial ring $A = R[x]$ is right (resp. left) Archimedean. Then $Z_r(A) = N(A) \subseteq J(A)$ (resp. $Z_l(A) = N(A) \subseteq J(A)$).
\end{corollary}

\begin{proof} Assume $R$ is right Archimedean. The containment $Z_r(A) \subseteq N(A)$ follows from Proposition \ref{1987}, and since the opposite containment is obvious, we obtain $Z_r(A) = N(A)$. Now applying Lemma \ref{2211}(b) completes the proof of the right Archimedean case. The left Archimedean case follows by an analogous argument.
\end{proof}

In the following two theorems we characterize reduced right (resp. left) Archime\-de\-an skew polynomial rings.

\begin{theorem} \label{1130}
Let $R$ be a ring and $\alpha$ an endomorphism of $R$. Then the following conditions are equivalent:
\begin{enumerate}
\item[(i)] $R[x; \alpha]$ is a reduced right Archimedean ring.
\item[(ii)] $R[x; \alpha]$ is a right Archimedean domain.
\item[(iii)] $R$ is a right Archimedean domain, $\alpha$ is injective and $\alpha$ preserves nonunits of~$R$.
\end{enumerate}
\end{theorem}

\begin{proof} (i) $\Leftrightarrow$ (ii): Assume (i) and suppose,  for contradiction, that $R[x; \alpha]$ is not a domain. Then there exists a nonzero right zero-divisor $f \in R[x; \alpha]$. Since $R[x; \alpha]$ is reduced and right Archimedean, Proposition \ref{1987} implies $fx = 0$ and thus $f = 0$, a contradiction. Hence $R[x; \alpha]$ is a right Archimedean domain and the proof of the implication (i) $\Rightarrow$ (ii) is complete. The opposite implication is obvious. 

(ii) $\Leftrightarrow$ (iii): This equivalence follows from Theorem \ref{2721}(a).
\end{proof}

\begin{theorem} \label{1144}
Let $R$ be a ring and $\alpha$ an endomorphism of $R$. Then the following conditions are equivalent:
\begin{enumerate}
\item[(i)] $R[x; \alpha]$ is a reduced left Archimedean ring.
\item[(ii)] $R[x; \alpha]$ is a left Archimedean domain.
\item[(iii)] $R$ is a left Archimedean domain and $\alpha$ is injective.
\end{enumerate}
\end{theorem}

\begin{proof} The equivalence (i) $\Leftrightarrow$ (ii) can be proved similarly as the equivalence (i) $\Leftrightarrow$ (ii) in Theorem \ref{1130}. The equivalence (ii) $\Leftrightarrow$ (iii) follows from Theorem \ref{2721}(b).
\end{proof}

As an immediate consequence of Theorems \ref{1130} and \ref{1144} we obtain the following corollary.

\begin{corollary} \label{1177}
For any ring $R$ the following conditions are equivalent:
\begin{enumerate}
\item[(i)] $R[x]$ is a reduced right (resp. left) Archimedean ring.
\item[(ii)] $R[x]$ is a right (resp. left) Archimedean domain.
\item[(iii)] $R$ is a right (resp. left)  Archimedean domain.\end{enumerate}
\end{corollary}

Looking at Corollary \ref{1177} one could conjecture that for a one-sided Archimedean ring, to be reduced is equivalent to be a domain. However, this is not the case, as we will see in Example \ref{1234}.

\section{Reduced Archimedean skew power series rings} \label{4}

This section is devoted to characterizations of skew power series rings that are reduced and right or left Archimedean. In this context the concept of a rigid endomorphism appears naturally, as we can see in Proposition \ref{2860} below. Recall from \cite{Krempa} that an endomorphism $\alpha$ of a ring $R$ is said to be \textit{rigid} if $a\alpha (a)=0$ implies $a = 0$, for any $a \in R$. Note that if $\alpha$ is rigid, then $\alpha$ is injective and $R$ is reduced (see \cite[p. 218]{Hong}).

\begin{proposition} \label{2860} {\rm (Krempa \cite[Corollary 3.5]{Krempa})}  
Let $R$ be a ring and $\alpha$ an endomorphism of $R$. Then the skew power series ring $R[[x; \alpha]]$ is reduced if and only if $\alpha$ is rigid.
 \end{proposition}
 
 Recall from \cite{Annin} that an endomorphism $\alpha$ of a ring $R$ is said to be \emph{compatible}  if for any $a, b \in R$ we have $ab = 0$ if and only if $a\alpha(b) = 0$.  We will need the following well-known relationship between rigid and compatible endomorphisms (see, e.g., \cite [Lemma 2.2]{Hashemi}).
 
 \begin{lemma} \label{compatible}
Let $R$ be a ring and $\alpha$ an endomorphism of $R$. Then $\alpha$ is rigid if and only if $\alpha$ is compatible and $R$ is reduced.
\end{lemma}

We will also need the following observation.

\begin{lemma} \label{9000}
Let $R$ be a ring and $\alpha$ a rigid endomorphism of $R$. Then for any posi\-ti\-ve integer $n$, any  permutation $\sigma$ of the set $\{1, \ldots, n \}$, any elements $a_1, \ldots, a_n \in R$ and any nonnegative integers $t_1, \ldots, t_n$ and $k_1, \ldots, k_n$,
$$\alpha^{t_1}(a_{1}^{k_1}) \alpha^{t_2}(a_{2}^{k_2}) \cdots \alpha^{t_n}(a_{n}^{k_n}) = 0 \ \text{\it if and only if } a_{\sigma(1)}a_{\sigma(2)} \cdots a_{\sigma(n)} = 0.$$ 
\end{lemma}

\begin{proof} It suffices to show that the following two equivalences hold:
%\begin{equation} \label{9186}
$$\alpha^{t_1}(a_{1}^{k_1}) \alpha^{t_2}(a_{2}^{k_2}) \cdots \alpha^{t_n}(a_{n}^{k_n}) = 0 \Leftrightarrow a_{1}^{k_1} a_{2}^{k_2} \cdots a_{n}^{k_n} = 0 \Leftrightarrow a_{\sigma(1)}a_{\sigma(2)} \cdots a_{\sigma(n)} = 0.$$
%\end{equation}

Since $\alpha$ is rigid, $\alpha$ is compatible and $R$ is reduced by Lemma \ref{compatible}. Thus the first of the above equivalences follows from \cite[Lemma 3.1]{Chen}, and the second one is an immediate consequence of \cite[Lemma 1.2]{Krempa}.
\end{proof}

Now we are ready to characterize skew power series rings that are reduced and right Archime\-de\-an.

\begin{theorem} \label{1133}
Let $R$ be a ring and $\alpha$ an endomorphism of $R$. Then the following conditions are equivalent:
\begin{enumerate}
\item[(i)] $R[[x; \alpha]]$ is a reduced right Archimedean ring.
\item[(ii)] $R$ is a right Archimedean ring, $\alpha$ is rigid and $\alpha$ preserves nonunits of~$R$.
\item[(iii)] $R$ is a reduced right Archimedean ring, $\alpha$ is compatible and $\alpha$ preserves non\-units of~$R$.
\end{enumerate}
\end{theorem}

\begin{proof} (i) $\Rightarrow$ (ii): Assume (i). Then $R$ is right Archimedean by Proposition \ref{5500}, and $\alpha$ is rigid by Proposition \ref{2860}. Let $a$ be a nonunit of $R$ and suppose $\alpha(a) \in U(R)$. Since $xa = \alpha(a)x$, we obtain $x = \alpha(a)^{-1}xa$ and thus Lemma \ref{2211}(a) implies $a \in U(R)$, a contradiction. Hence $\alpha$ preserves nonunits of $R$.

(ii) $\Rightarrow$ (i): Denote $A = R[[x; \alpha]]$. Since $\alpha$ is rigid, $A$ is reduced by Proposition \ref{2860}. To prove that $A$ is right Archimedean, let $f = \sum_{i = 0}^{\infty} f_ix^i \in A$ and $g = \sum_{i = 0}^{\infty} g_ix^i \in A$ be such that $f \in \bigcap_{n \in \mathbb{N}} A g^n$. Thus for any $n \in \mathbb{N}$ there exists $h_n = \sum_{i = 0}^{\infty} h_{i}^{(n)} x^i \in A$ such that $f = h_ng^n$. If $g_0 \in U(R)$ then $g \in U(A)$. Hence we assume that $g_0 \not\in U(R)$, and to prove that $A$ is right Archimedean, it suffices to show that $f = 0$. We will prove by induction on $m$ that for every $m \geq 0$ the following condition (\ref{7812}) holds for all $k \leq m$, which obviously implies $f = 0$:
\begin{equation} \label{7812}
f_k = h_{k}^{(n)} g_0 = 0 \ \  \text{\it for any $n > k$}.
\end{equation}

We prove first that (\ref{7812}) holds for $m = 0$. In this case, (\ref{7812}) reduces to 
\begin{equation} \label{1173}
f_0 = h_{0}^{(n)} g_0 = 0 \ \  \text{\it for any $n > 0$}.
\end{equation}
To prove (\ref{1173}), let $n > 0$. Equating constant terms on each side of $f = h_ng^n$ gives
\begin{equation} \label{8877}
f_0 = h_{0}^{(n)}g_{0}^{n},
\end{equation}
and thus $f_0 \in Rg_{0}^n$. Therefore, $f_0 \in \bigcap_{n \in \mathbb{N}} Rg_{0}^{n}$, and since $R$ is right Archimedean and $g_0 \not\in U(R)$, it follows that $f_0 = 0$. Hence by (\ref{8877}) we have $h_{0}^{(n)} g_{0}^{n}= 0$, and thus $h_{0}^{(n)} g_0 = 0$ by Lemma \ref{9000}, which proves (\ref{1173}), so also proves the $m = 0$ case.

Assume that $m \geq 1$ and (\ref{7812}) holds for any $k < m$. To complete the proof, it suffices to show that (\ref{7812}) holds for $k = m$, i.e.,
\begin{equation} \label{0010}
f_{m} = h_{m}^{(n)}g_{m} = 0 \ \  \text{\it for any $n > m$}.
\end{equation}
To prove (\ref{0010}), consider any $n > m$. Since $f = h_ng^n$, $f_{m}$ is the sum of $x^{m}$-coe\-ffi\-cients of all products of monomials
\begin{equation} \label{2185}
h_{j}^{(n)}x^j \cdot  g_{i_1}x^{i_1} \cdot g_{i_2} x^{i_2}  \cdots  g_{i_n}x^{i_n}
\end{equation}
with $j + i_1 + i_2 + \cdots + i_n = m$.
Hence $f_{m}$ is a sum of products of the form
\begin{equation}\label{5441}
h_{j}^{(n)} \alpha^{t_1}(g_{i_1}) \alpha^{t_2}(g_{i_2}) \cdots \alpha^{t_n}(g_{i_n}),
\end{equation}
where $j + i_1 + i_2 + \cdots + i_n = m$ and $t_1, \ldots, t_n$ are some nonnegative integers. Since $n > m$, at least one of the indexes $i_1, i_2, \ldots, i_n$ has to be equal to 0. Hence there exists $p \in \{1, 2, \ldots, n \}$ such that $g_{i_p} = g_0$. If $j < m$, then by the induction hypothesis we have $h_{j}^{(n)}g_p = h_{j}^{(n)}g_0 = 0$, and thus Lemma \ref{9000} implies that the product (\ref{5441}) is equal to 0 in this case. Hence the only case we are left with is that for $j = m$. Then $i_1 = i_2 = \ldots = i_n = 0$ and in this case the product (\ref{2185}) is equal to
\begin{equation} \label{5391}
h_{m}^{(n)}x^{m} g_{0}^{n} = h_{m}^{(n)}\alpha^{m}(g_{0}^{n})x^{m}.
\end{equation}
Hence for any $n > m $ we have
\begin{equation} \label{1930}
f_{m} = h_{m}^{(n)} \alpha^{m}(g_{0}^{n}) = h_{m}^{(n)} \alpha^{m}(g_{0})^{n},
\end{equation}
and thus $f_{m} \in \bigcap_{n \in \mathbb{N}} R(\alpha^{m}(g_{0}))^{n}$. Since $g_0 \not\in U(R)$ and $\alpha$ preserves nonunits, also $\alpha^{m}(g_0) \not\in U(R)$. Hence, since $R$ is right Archimedean, $f_{m} = 0$ follows. Thus, by combining Lemma \ref{9000} with (\ref{1930}), we obtain $h_{m}^{(n)}g_0 = 0$, so (\ref{0010}) holds, which completes the proof of the implication (ii) $\Rightarrow$ (i).

(ii) $\Leftrightarrow$ (iii): This equivalence follows from Lemma \ref{compatible}.
\end{proof}

In the result below we characterize skew power series rings that are reduced left Archimedean.

\begin{theorem} \label{1131}
Let $R$ be a ring and $\alpha$ an endomorphism of $R$. Then the following conditions are equivalent:
\begin{enumerate}
\item[(i)] $R[[x; \alpha]]$ is a reduced left Archimedean ring.
\item[(ii)] $R$ is a left Archimedean ring and $\alpha$ is rigid.
\item[(iii)] $R$ is a reduced left Archimedean ring and $\alpha$ is compatible.
\end{enumerate}
\end{theorem}

\begin{proof} The proof is similar to that of Theorem \ref{1133}.   In the left Archimedean case, instead of  (\ref{5391}) we obtain the monomial $g_{0}^{n}h_{m}^{(n)}x^{m}$ and thus (\ref{1930}) changes to $f_{m} = g_{0}^{n}h_{m}^{(n)}$. Hence in this case, to follow the argument of the proof of Theorem \ref{1133}, there is no need to assume that $\alpha$ preserves nonunits of $R$.
\end{proof}

As an immediate consequence of Theorems \ref{1133} and \ref{1131} we obtain the following corollary.

\begin{corollary}  \label{5100}
For any ring $R$ the following conditions are equivalent:
\begin{enumerate}
\item[(i)] $R[[x]]$ is a reduced right (resp. left) Archimedean ring.
\item[(ii)] $R$ is a reduced right (resp. left)  Archimedean ring.\end{enumerate}
\end{corollary}

As promised, we will present an example of a reduced  Archimedean ring which is not a domain. Our example will be based on the following observation.

\begin{proposition} \label{2381}
Let $R$ be a ring and $I_1, I_2$ ideals of $R$.
\begin{enumerate}
\item[(a)] If the rings $R/I_1$ and $R/I_2$ are reduced, then so is the ring $R/(I_1 \cap I_2)$.
\item[(b)] If the ideals $I_1$ and $I_2$ are incomparable (i.e., $I_1 \not\subseteq I_2$ and $I_2 \not\subseteq I_1$), then the ring $R/(I_1 \cap I_2)$ is not a domain.
\item[(c)] If $I_1, I_2 \subseteq J(R)$ and both the rings $R/I_1$ and $R/I_2$ are right (resp. left) Archimedean, then the ring $R/(I_1 \cap I_2)$ is right (resp. left) Archimedean.
\end{enumerate}
\end{proposition} 

\begin{proof} For any $r \in R$, let $\overline{r}$ denote the coset $r + I_1 \cap I_2$ in the factor ring $\overline{R} = R/(I_1 \cap I_2)$. 

(a): Let $a \in R$ be such that $\overline{a}^2 = \overline{0}$. Then $a ^2 \in I_1$ and $a^2 \in I_2$, and since $R/I_1$ and $R/I_2$ are reduced, it follows that $a \in I_1 \cap I_2$, i.e., $\overline{a} = \overline{0}$. Hence $\overline{R}$ is reduced.

(b): Since $I_1, I_2$ are incomparable, there exist $a \in I_1 \setminus I_2$ and $b \in I_2 \setminus I_1$. Then $ab \in I_1 \cap I_2$, so we have $\overline{a} \neq \overline{0}$, $\overline{b} \neq \overline{0}$ and $\overline{a} \overline{b} = \overline{0}$. Hence $\overline{R}$ is not a domain.

(c): We consider only the right Archimedean case, since the left Archimedean case follows by similar arguments.

To prove that $\overline{R}$ is right Archimedean, let $b, a \in R$ be such that $\overline{b} \neq \overline{0}$ and $\overline{b} \in \bigcap_{n \in \mathbb{N}} \overline{R} \overline{a}^n$, i.e., $b \not\in I_1 \cap I_2$ and for any $n \in \mathbb{N}$ there exists $c_n \in R$ such that $b - c_na^n \in I_1 \cap I_2$.  Without loss of generality we can assume that $b \not\in I_1$. Then, since $b - c_na^n \in I_1$ and $R/I_1$ is right Archimedean, the coset $a + I_1$ is a unit of $R/I_1$, i.e., there exists $v \in R$ such that $av - 1 \in I_1$ and $va - 1 \in I_1$. Hence $av, va \in 1 + I_1$, and since $I_1 \subseteq J(R)$, we obtain $av, va \in U(R)$, which implies $a \in U(R)$. Thus in the ring $\overline{R}$ we have $\overline{a} \in U(\overline{R})$, which shows that $\overline{R}$ is right Archimedean. 
\end{proof}

\begin{example} \label{1234} ({\it There exists a commutative reduced Archimedean ring which is not a domain.}) 
Let $F$ be a field and $R = F[[x, y]]$ the ring of power series in the variables $x, y$ over $F$. Let $I_1$ (resp. $I_2$) be the ideal of $R$ generated by $x$ (resp. $y$). Then $R/I_1 \simeq F[[y]]$ and $R/I_2 \simeq F[[x]]$ and thus by Corollary \ref{5100}  both the rings $R/I_1$ and $R/I_2$ are reduced Archimedean. Since the Jacobson radical of $R$ consists of those power series whose constant term is zero, $I_1 \subseteq J(R)$ and $I_2 \subseteq J(R)$, and obviously the ideals $I_1, I_2$ are incomparable. Hence by Proposition \ref{2381} the factor ring $R/(I_1 \cap I_2)$ is a commutative reduced Archimedean ring which is not a domain.
\end{example}

We close the paper with an example of a reduced right and left Archimedean skew power series ring which is not a domain.

\begin{example} \label{7700} ({\it There exists a reduced skew power series ring which is left and right Archimedean, but neither a domain nor a commutative ring}). Let $R$ be the ring from Example~\ref{1234}. Since $R$ is the factor ring of $F[[x, y]]$ modulo the ideal generated by the product $xy$, we can consider the ring $R$ to be the set of elements uniquely written in the form $a + B(x) + C(y)$, where $a \in F$, $B(x) \in xF[[x]]$ and $C(y) \in yF[[y]]$, with usual addition and with multiplication subject to the relation $xy = 0$. We define a map $\alpha: R \rightarrow R$ by setting $\alpha(a + B(x) + C(y)) = a + B(x^2) + C(y)$. One can easily verify that $\alpha$ is a rigid  endomorphism of $R$ preserving nonunits of $R$. Hence by Theorems \ref{1133} and \ref{1131}, the skew power series ring $R[[t; \alpha]]$ is a reduced ring which is left and right Archimedean, but neither a domain nor a commutative ring.
\end{example}

\section*{Acknowledgments}

This work was supported by the grant WZ/WI/1/2019 from the Bialystok University of Technology and funded from the resources for research by the Ministry of Science and Higher Education of Poland.


\begin{thebibliography}{99}

\bibitem{AAZ}
D. D. Anderson, D. F. Anderson and M. Zafrullah, Completely integrally closed Pr\"ufer v-multiplication domains, {\it Comm. Algebra} 45 (2017), no. 12, 5264-5282. 

\bibitem{Annin}
S. Annin, Associated primes over skew polynomial rings, \emph{Comm. Algebra} 30 (2002), no. 5, 2511-2528. 

\bibitem{BD}
R. A. Beauregard and D. E. Dobbs, On a class of Archimedean integral domains, {\it Canad. J. Math.} 28 (1976), no. 2, 365-375. 

\bibitem{Chen}
W. X. Chen and S. Y. Cui, On weakly semicommutative rings, \emph{Comm. Math. Res.} 27 (2011), no. 2, 179-192.

\bibitem{Coykendall}
J. Coykendall and T. Dumitrescu, An infinite product of formal power series, 
{\it Bull. Math. Soc. Sci. Math. Roumanie (N.S.)} 49(97) (2006), no. 1, 31-36. 

\bibitem{Dobbs}
D. E. Dobbs, Prime ideals surviving in complete integral closures, {\it Arch. Math. (Basel)} 69 (1997), no. 6, 465-469. 

\bibitem{Dumitrescu}
T. Dumitrescu, S. O. I. Al-Salihi, N. Radu and T. Shah, Some factorization properties of composite domains $A+XB[X]$ and $A+XB[[X]]$, {\it Comm. Algebra} 28 (2000), no. 3, 1125-1139.

\bibitem{GR1990}
S. Gabelli and M. Roitman, On Nagata's theorem for the class group, \emph{J. Pure Appl. Algebra} 66 (1990), no. 1, 31-42. 

\bibitem{GR}
S. Gabelli and M. Roitman, On finitely stable domains I, {\it J. Commut. Algebra} 11 (2019), no.~1, 49-67. 

\bibitem{GRII}
S. Gabelli and M. Roitman, On finitely stable domains II, {\it J. Commut. Algebra} 12 (2020), no. 2, 179-198. 

\bibitem{Gilmer}
R. Gilmer, A note on the quotient field of the domain $D[[x]]$, {\it Proc. Amer. Math. Soc.}
18 (1967), 1138-1140.

\bibitem {Hashemi} E. Hashemi and A. Moussavi, Polynomial extensions of quasi-Baer rings,
\emph{Acta Math. Hungar.} 107 (2005), no. 3, 207-224.

\bibitem{Heinzer} 
W. Heinzer, K. A. Loper, B. Olberding, H. Schoutens and M. Toeniskoetter, Ideal theory of infinite directed unions of local quadratic transforms, \emph{J. Algebra}  474 (2017), 213-239. 

\bibitem{Hong}
C. Y. Hong, N. K. Kim and T. K. Kwak, Ore extensions of Baer and p.p.-rings, \emph{J. Pure Appl. Algebra} 151 (2000), no. 3, 215-226.

\bibitem{Krempa}
J. Krempa, Some examples of reduced rings, {\it Algebra Colloq.} 3 (1996), no. 4, 289-300. 

\bibitem{RMArchim}
R. Mazurek, Archimedean domains of skew generalized power series,  {\it Forum Math.}  32 (2020), no. 4, 1075--1093.

\bibitem{MPQ}
H. Mousavi, F. Padashnik and A. A. Qureshi, On reduced Archimedean skew power series rings, preprint, arXiv:2009.01473v1

\bibitem{MPP} 
A. Moussavi, F. Padashnik and K. Paykan, Archimedean skew generalized power series rings, {\it Commun. Korean Math. Soc.} 34 (2019), no. 2, 361-374.

\bibitem{Nasr}
A. R. Nasr-Isfahani, The ascending chain condition for principal left ideals of skew polynomial rings, {\it Taiwanese J. Math.} 18 (2014), no. 3, 931-941. 

\bibitem{Sheldon}
P. B. Sheldon, How changing $D[[x]]$ changes its quotient field, 
{\it Trans. Amer. Math. Soc.} 159 (1971),  223-244. 

\end{thebibliography}
\end{document}